\documentclass[12pt]{article}
\usepackage[mathscr]{eucal}
\usepackage{amssymb}
\usepackage{amsthm}
\usepackage{enumitem}
\usepackage{array}
\usepackage[margin=1in]{geometry}
\newlist{steps}{enumerate}{1}
\setlist[steps, 1]{label = Step \arabic*:}
\usepackage{mathtools}
\usepackage{xcolor}
\usepackage{comment}

\theoremstyle{plain}
\newtheorem{theorem}{Theorem}[section]
\newtheorem{proposition}[theorem]{Proposition}
\newtheorem{lemma}[theorem]{Lemma}
\newtheorem{corollary}[theorem]{Corollary}

\theoremstyle{definition}
\newtheorem{definition}[theorem]{Definition}
\newtheorem{example}[theorem]{Example}

\theoremstyle{remark}
\newtheorem{remark}[theorem]{Remark}
\numberwithin{equation}{section}

\usepackage[colorlinks=true,linkcolor=blue,citecolor=blue,urlcolor=blue]{hyperref}

\begin{document}
		 
	\begin{center} {\large{\bf Projections in the Algebra generated by an $n$-Potent Operator}}	
    
		\bigskip
		{\sc Monika$^1$, Priyadarshi Dey$^2$, Zachary Easley$^3$}
		
		\medskip
		
		$^1$Hampton University, Hampton, VA, USA\\
		$^2$ Millsaps College, Jackson, Mississippi, USA\\
        $^3$ University of Memphis, Memphis, Tennessee, USA

		
		\medskip
		fnu.monika@hamptonu.edu, deyp@millsaps.edu, zkeasley@memphis.edu

	\end{center}
	
	\medskip\noindent
    \begin{abstract}
This paper investigates the projection operators that lie in the algebra generated by powers of an $n$-potent operator $T$ on a complex Banach space, where $T^n = T$. We give a complete description of all projections in the algebra $\operatorname{comb}(T) = \text{span}\{T, T^2, \dots, T^{n-1}\}$, and prove that each such projection is uniquely determined by, and in bijection with, a subset of the nonzero spectrum of $T$. As a consequence, the family of projections in $\operatorname{comb}(T)$ forms a Boolean algebra isomorphic to the power set of $\sigma(T)\setminus\{0\}$. We also establish a spectral decomposition for $n$-potent operators in terms of their Riesz projections and derive explicit formulas for the associated Riesz projections using resolvent expansions. We give an illustration of the theory for $5$-potent operators, which highlights the algebraic and spectral structure of finite-order operators on Banach spaces.

\end{abstract}
\vspace{1em}
\medskip


\section{Introduction}

The structure of projections plays a fundamental role in operator theory has due to its connection with spectral theory, invariant subspace decomposition, and ergodic theory. In this paper we study operators satisfying $T^n=T$ for some positive integer $n \ge 2$, called $n$-potent operators, arise in several contexts: as periodic transformations in dynamical systems \cite{EFHN}, as operators with finite spectrum in the study of invariant subspaces \cite{RR}, and as generalizations of projections (the case $n=2$).

Previous work by Bikchentaev and Yakushev \cite{BY} provided representations for tri-potent operators ($n=3$) and showed connections to convex combinations of isometries. More recently, Ili\v{s}evi\'{c}, Li, and Poon \cite{ILP} studied generalized circular projections, which share structural similarities with the spectral decompositions of $n$-potent operators. Along the same line, projections in the convex hull of isometries on function spaces have been studied in \cite{BDE,HJ}, with particular focus on absolutely continuous function spaces. For cyclic group of isometries of finite order, the unique projection in the convex hull is given by the average of all group elements \cite{BDE}. This connects to $n$-potent operators in the special case where $T$ is periodic (see \autoref{MainThm1}(iii)).

While the spectral structure of $n$-potent operators is well understood, a systematic description of projection operators belonging to the algebra generated by $\{T,T^2, T^3, \dots, T^{n-1}\}$ has not been explicitly characterized. Such a characterization is of interest for several reasons: it provides a complete description of invariant subspace decompositions constructable from the powers of $T$ \cite{RR}, yields explicit computational formulas for spectral projections without evaluating the contour integration, and for dynamical system applications \cite{EFHN}.

Throughout this paper, $X$ denotes a complex Banach space and $\mathcal{B}(X)$ the algebra of bounded linear operators on $X$. For an $n$-potent operator $T \in \mathcal{B}(X)$, we consider the finite-dimensional subalgebra
\[\operatorname{comb}(T):= \operatorname{span}\{T,T^2,\dots,T^{n-1}\}=\left\{\sum_{i=1}^{k}a_iT^{i}: a_i \in \mathbb{C}, 1\le k \le n-1\right\}.\]
Since $T^n=T$, every polynomial in $T$ reduces to an element of $\operatorname{comb}(T)$. 

Our main result provides a complete characterization of projections in $\operatorname{comb}(T)$. We show that each such projection is uniquely determined by a subset of the nonzero spectrum of $T$ (\autoref{thm:proj-conv-comb-result}), yielding a Boolean algebra of projections isomorphic to $\mathcal{P}(\sigma(T) \setminus \{0\})$ (Corollary~\ref{boolean-algebra-isomorphism}). The paper also gives explicit formulas for the associated Riesz projections (\autoref{thm:spectrumandeigenprojection}) and illustrates the theory through finite and infinite-dimensional examples, including a complete analysis of the $5$-potent case (Section~\ref{sec3-examples}).

\section{Preliminaries for $n$-potent Operators}\label{sec1:prelim}

In this section we establish the fundamental spectral properties of $n$-potent operators and also develop theory of Riesz projections that will be essential for this paper. We show that the spectrum of $n$-potent operators is finite and consists of semisimple eigenvalues.  
\begin{definition}
A bounded linear operator $T$ on $X$ is said to be an $n$-potent operator, or simply $n$-potent, if there exists an integer $n \geq 2$ such that $T^n = T.$  
 \end{definition}
 Evidently, if $T$ is projection on $X$, then $T$ is a $2$-potent operator on $X$ (also referred to as an ``idempotent"). Similarly, if $T$ is a periodic operator with period $m$ (that is, $T^m=I$), then $T$ is $(m+1)$-potent. 
\medskip

In what follows, we will always assume that $n$ is an integer, at least $2$, and that $T \in \mathcal{B}(X) \setminus \{ 0, I \}$, where $I$ is the identity operator on $X$. We denote the spectrum of $T$ by $\sigma(T)$ and the point spectrum of $T$ by $\sigma_p (T)$. We recall that if $\lambda \in \sigma_p (T)$, then $T - \lambda I$ has non-trivial Kernel, i.e., $\lambda$ is an eigenvalue for $T$. We set $\mu_n := \{e^{\frac{2 \pi i k}{n}} : k = 0,\dots,n-1 \}$ to denote the complete set of $n^{th}$ roots of unity.
\medskip

Let $\lambda$ be an eigenvalue for $T$. The generalized eigenspace for $T$ associated with $\lambda$ is defined as:
   
        \[E_{\lambda} := \bigcup_{k=1}^{\infty} \operatorname{Ker}((T - \lambda I)^k).\]
   
The dimension of $E_{\lambda}$, when finite, is called the algebraic multiplicity of $\lambda$, while $\operatorname{dim}(\operatorname{Ker}(T-\lambda I))$ is the geometric multiplicity for $\lambda$. The eigenvalue $\lambda$ is said to be semisimple if $\operatorname{Ker}((T- \lambda I)^2)=\operatorname{Ker}(T- \lambda I)$ and is said to be simple if $\operatorname{dim}(\operatorname{Ker}(T-\lambda I))=1$.

\medskip

An operator $T$ on $X$ is said to be an \textit{algebraic operator} (see e.g. \cite{IK}, \cite{TK}) if there exists a polynomial $p(t)$ with coefficients in $\mathbb{C}$ so that $p(T)x = 0$ for all $x \in X$. We refer to such a polynomial $p(t)$ as an \textit{annihilating polynomial} for $T$. The \textit{minimal} polynomial, $m(t)$, for an operator $T \in \mathcal{B}(X)$ is the monic polynomial of generates the principal ideal $\mathfrak{P} := \{ p \in \mathbb{C}[t] : p(T)x=0 \, \, \mathrm{for \ all} \, \, x \in X \}.$ It is worthy to note that every algebraic operator $T \in \mathcal{B}(X)$ has a unique minimal polynomial, since $\mathfrak{P}$ forms a non-zero ideal in the principal ideal ring $\mathbb{C}[z]$. On the other hand, if $T$ is an $n$-potent operator, then $T$ is algebraic with an annihilating polynomial $p(z)=z^n-z$. 

\medskip
The following Proposition establishes that the spectrum of an $n$-potent operator is finite and is equal to the point-spectrum of the operator (i.e., the spectrum comprises only of eigenvalues) and hence each eigenvalue is isolated.

\begin{proposition}
\label{SpecOfnPotentOpProp}
    Let $T$ be an $n$-potent operator on $X$. Then, the spectrum of $T$ is finite and consists of semisimple eigenvalues of $T$. In particular, $\sigma(T) \subseteq \mu_{n-1} \cup \{ 0\}$ and $\sigma(T) = \sigma_p(T)$. 
\end{proposition}
\begin{proof}
    If $T$ is an $n$-potent operator, then the polynomial $p(t) = t^n - t$ is an annihilating polynomial for $T$. By the Spectral Mapping Theorem (\cite{GPR}, p. 263), since $p(T) = 0$, we see that $\sigma(p(T)) = 0 = p (\sigma(T))$, implying that each element of the spectrum of $T$ is a root of $p(t)$.
    Since the roots of $p(t)$ are among the $(n-1)^{th}$ roots of unity  or $0$, it follows that $\sigma(T) \subseteq \mu_{n-1} \cup \{ 0 \}$, and in particular, $\sigma(T)$ is finite with cardinality at most $n$.
    \medskip
    
 To show that $\sigma(T) = \sigma_p(T)$, it is enough to show $\sigma(T) \subseteq \sigma_p(T)$. To this end, let $\lambda \in \sigma(T)$. Since $T$ is algebraic, let $p(z) \in \mathbb{C}[z]$ be the minimal polynomial for $T$. Then there exists a polynomial $q \in \mathbb{C}[z]$ such that $p(T)=(T-\lambda I)^mq(T)=0$ where $q(\lambda) \neq 0$ and $m \geq 1$ is the multiplicity of $\lambda$. It implies that $Ker\left((T-\lambda I)^m\right) \neq \{0\}$ and let $x \in Ker\left((T-\lambda I)^m\right)$. Define a non-zero vector $v:= (T-\lambda I)^{m-1}x.$ It is easy to check that, $Tv=\lambda v$, which implies $\lambda \in \sigma_{p}(T).$ To finish the proof, it remains to show that $\lambda \in \sigma(T)$ is a semisimple eigenvalue of $T$; that is $\operatorname{Ker}((T-\lambda I)^2) \subseteq \operatorname{Ker}(T-\lambda I)$. Let $y \in \operatorname{Ker}((T-\lambda I)^2)$, which implies, $(T-\lambda I)\big((T-\lambda I)(y)\big)=0$; set $x:=(T-\lambda I)(y).$ We will prove that $x=0$, that is $y \in \operatorname{Ker}(T-\lambda I)$. Since $p(t)=t(t^{n-1}-1)$ is an annihilating polynomial for $T$ and has no repeated roots, hence, $p(t)= (t-\lambda)q(t),$ for some $q(z) \in \mathbb{C}[z]$ with $q(\lambda) \ne 0$. Then $p(T)= (T-\lambda I)q(T)=0$. We evaluate $p(T)$ at $y \in X$ and use commutativity to get $q(T)x=0$. Since $x \in \operatorname{Ker}(T-\lambda I),$ we get 
 \[q(T)x=q(\lambda)x=0.\]
 But since $q(\lambda)\ne 0$, then $x=0$ and hence $\lambda$ is semisimple; which completes the proof of the Proposition. 
\end{proof}

We recall the definition of the following operator, which is very central to our discussion for the remainder of the paper. It is called the Riesz projection, see \cite{GGK} for more details. 
\begin{definition}(see \cite{GGK}, p.~9)
   \label{DefOfEigenProj}
   Let $T \in \mathcal{B}(X)$. Suppose that the spectrum of $T$, $\sigma(T)$, is the disjoint union of two non-empty closed subsets $\sigma$ and $\tau$, and let $\Gamma$ be a Cauchy contour in the resolvent set of $T$ enclosing $\sigma$ and separating it from $\tau = \sigma(T) \setminus \sigma$. The operator defined by
   \begin{align}
       \label{eigenprojeqforsigma}
       P_{\sigma} := -\frac{1}{2 \pi i} \int_{\Gamma} (T - z I)^{-1} \, dz
   \end{align}
   is called the Riesz projection of $T$ corresponding to $\sigma$.
\end{definition}
In practice, we are particularly interested in computing the Riesz projection associated to a single point in the spectrum of $T$, that is, when  $\sigma = \{ \lambda \}$. In this case, we denote the projection either by $P_{\{\lambda\}}$ or $P_{\lambda}$.

Since the operator $(T-zI)^{-1}$ is analytic on $\rho(T)$, it can be shown that the integral in (\ref{eigenprojeqforsigma}) does not depend on $\Gamma$. The next proposition shows several important properties of the Riesz projection as defined in  the Definition \ref{DefOfEigenProj}. We refer the reader to \cite{GGK} for further details about the Riesz projections. 

We recall the following theorem from \cite{GGK}, which will be useful in establishing key properties of the Riesz projection. The proof is omitted here and can be found in \cite[Theorem~2.2, p.~10]{GGK}. 

\begin{theorem}\label{theorem:inv-subspace}
Let $T \in \mathcal{B}(X)$, and let $\sigma$ and $\tau = \sigma(T) \setminus \sigma$ be two disjoint, non-empty, closed subsets of the spectrum $\sigma(T)$. Define $R_{\sigma} = \operatorname{Ran}(P_\sigma)$ and $K_{\sigma} = \operatorname{Ker}(P_\sigma)$. Then:
\begin{itemize}
    \item $X = R_{\sigma} \oplus K_{\sigma}$,
    \item $R_{\sigma}$ and $K_{\sigma}$ are $T$-invariant subspaces, and
    \item $\sigma(T|_{R_\sigma}) = \sigma$, \quad $\sigma(T|_{K_\sigma}) = \sigma(T) \setminus \sigma$.
\end{itemize}
\end{theorem}

We are now ready to prove the following proposition. While parts of the proof can be found in \cite{GGK} and \cite{TK}, we include it here for the sake of completeness. 

 \begin{proposition}
\label{basicpropsofeigenprojectionsprop}
    Let $T \in \mathcal{B}(X)$, and let $\lambda \in \sigma(T)$ be an isolated point of the spectrum. Denote $\tau := \sigma(T) \setminus \{\lambda\}$. Then the following statements about the Riesz projection $P_{\{\lambda\}}$ hold:
    \begin{enumerate}[label=(\roman*)]
    \item $P_{\{\lambda\}}$ is linear and bounded. 
    \vspace{-0.2cm}
    \item $P_{\{\lambda\}}$ is a projection, i.e., $P_{\{\lambda\}}^2 = P_{\{\lambda\}}$. 
    \vspace{-0.2cm}
    \item $P_{\{\lambda\}}$ commutes with $T$, i.e., $P_{\{\lambda\}}T = TP_{\{\lambda\}}$.
    \vspace{-0.2cm}
    \item The projections $P_{\{\lambda\}}$ and $P_{\tau}$ form a resolution of the identity and are orthogonal, i.e., 
    \[
    P_{\{\lambda\}} + P_{\tau} = I \quad \text{and} \quad P_{\{\lambda\}} P_{\tau} = P_{\tau} P_{\{\lambda\}} = 0.
    \]
    \vspace{-1cm}
    \item Assume, in addition that, $T$ is algebraic. If $\lambda$ is a semisimple eigenvalue of $T$, then $P_{\{\lambda\}}$ is the eigenprojection onto $\operatorname{Ker}(T-\lambda I)$.
  \end{enumerate}
\end{proposition}
\begin{proof}
We proceed to prove the statements as follows. 
\begin{enumerate}[label=(\roman*)]
    \item The linearity of $P_{\{\lambda\}}$ follows from the linearity of the resolvent $(T - zI)^{-1}$ and the integral operator. To show boundedness, let $\Gamma$ be a Cauchy contour (compact and positively oriented) in the resolvent set of $T$ enclosing $\lambda$. For any $x \in X$, continuity of $z \mapsto (T - zI)^{-1}$ on $\Gamma$ implies there exists $M > 0$ such that $\|(T - zI)^{-1}\| \leq M$ on $\Gamma$. Then:
\[
\|P_{\{\lambda\}}(x)\| \leq \frac{1}{2\pi}M\,\text{length}(\Gamma)\cdot \|x\|,
\]
which shows $P_{\{\lambda\}}$ is bounded.
   
\item Let $\Gamma$ and $\gamma$ be  Cauchy contours around $\lambda$ separating $\lambda$ from $\sigma(T)\setminus \{\lambda\}$. We assume that $\Gamma$ is in the interior of $\gamma$. Then:
\begin{align*}
    P_{\{\lambda\}}^2&= \left(-\frac{1}{2\pi i}\int_{\Gamma}(T-zI)^{-1}\, dz\right)\left(-\frac{1}{2\pi i}\int_{\gamma}(T-\omega I)^{-1}\, d\omega\right)\\ &= \left(\frac{1}{2\pi i}\right)^2\int_{\Gamma}\int_{\gamma} (T-zI)^{-1}(T-\omega I)^{-1}dz d\omega
\end{align*}
Using the resolvent identity (c.f. equation (5) in \cite{TK}, pg 36), we have  
\begin{align*}
    P_{\{\lambda\}}^2 &= \left(\frac{1}{2\pi i}\right)^2\int_{\Gamma}\int_{\gamma}\frac{1}{z-\omega} (T-zI)^{-1}dz d\omega - \left(\frac{1}{2\pi i}\right)^2\int_{\Gamma}\int_{\gamma}\frac{1}{z-\omega} (T-\omega I)^{-1}dz d\omega
\end{align*}
For the notational convenience we write $P_{\{\lambda\}}^2= Q_1 - Q_2$, where 
\begin{align*}
    Q_1 &= \left(\frac{1}{2\pi i}\right)^2\int_{\Gamma}\int_{\gamma}\frac{1}{z-\omega} (T-zI)^{-1}dz d\omega\\ &= \left(\frac{1}{2\pi i}\right)^2\int_{\Gamma}(T-zI)^{-1}\left( \int_{\gamma} \frac{1}{z-\omega}\, d\omega\right)\,dz\\& = \left(\frac{1}{2\pi i}\right)^2\int_{\Gamma}(T-zI)^{-1}\, dz = P_{\{\lambda\}}.
\end{align*}
On the other hand, 
\begin{align*}
    Q_2 &= \left(\frac{1}{2\pi i}\right)^2\int_{\Gamma}\int_{\gamma}\frac{1}{z-\omega} (T-\omega I)^{-1}dz d\omega \\&= \left(\frac{1}{2\pi i}\right)^2\int_{\gamma}\int_{\Gamma}\frac{1}{z-\omega} (T-\omega I)^{-1}d\omega dz\\&= \left(\frac{1}{2\pi i}\right) \int_{\gamma}(T-\omega I)^{-1}\left( \frac{1}{2\pi i}\int_{\Gamma}\frac{1}{z-\omega}\, dz\right)\, d\omega
\end{align*}
Since, $\Gamma$ is in the interior of $\gamma$, so $\displaystyle\int_{\Gamma}\frac{1}{z-\omega}\, dz=0$, and therefore, $P_{\{\lambda\}}^2=P_{\{\lambda\}}$, which proves $(ii)$.

\item To prove the commutativity, we first note that for every $z \in \rho(T)$,
\begin{align*}
    T(T-zI)^{-1}&= (T-zI+zI)(T-zI)^{-1}=I+(T-zI)^{-1}z=(T-zI)^{-1}T.
\end{align*}
Then by a simple computation we can show that $P_{\{\lambda\}}T=TP_{\{\lambda\}}$.

\item We will first prove that the projections $P_{\{\lambda\}}$ and $P_{\tau}$ form a resolution of identity. To do this, let us choose Cauchy contours $\Gamma_{\lambda}, \Gamma_{\tau}$ and $\Gamma$ (in counterclockwise direction) such that $\Gamma_{\lambda}$ separates $\{\lambda\}$ from $\tau$ and $\Gamma$ encloses both the contours.
Then
\[
P_{\{\lambda\}} + P_{\tau} = -\frac{1}{2\pi i} \int_{\Gamma} (T - zI)^{-1} \, dz = P_{\sigma(T)}.
\]
Applying \autoref{theorem:inv-subspace} with $\sigma=\sigma(T)$, we see that $\sigma(T|_{K{_{\sigma(T)}}})$ is the empty set. Which implies $T|_{K_{\sigma(T)}}\equiv0$. Hence, $K_{\sigma(T)}=\{0\}$ and $X=\operatorname{Ran}(P_{\sigma(T)})$ and so $P_{\sigma(T)}=I$. Which proves $P_{\{\lambda\}} + P_{\tau} = I$. The orthogonality part follows from: 
\[P_{\{\lambda\}}P_{\tau}= P_{\{\lambda\}}(I-P_{\{\lambda\}})=0= P_{\tau}P_{\{\lambda\}}.\]
\item We now show that the subspace $R_{\{\lambda\}} = \operatorname{Ran}(P_{{\{\lambda\}}})$, as defined in \autoref{theorem:inv-subspace}, coincides with the subspace $\operatorname{Ker}(T-\lambda I)$, i.e., $R_{\{\lambda\}} = \operatorname{Ker}(T-\lambda I)$. We will first prove that $R_{\{\lambda\}}\subseteq \operatorname{Ker}(T-\lambda I)$. Since $\lambda$ is semisimple, we also have on $R_{\{\lambda\}}$, $\operatorname{Ker}((T-\lambda I)^2)=\operatorname{Ker}(T-\lambda I)$. To simplify the notations, let us write $S:= (T-\lambda I)\big|_{R_{\{\lambda\}}}$. By \autoref{theorem:inv-subspace}, we get $R_{\{\lambda\}}$ is a $T$-invariant subspace and $\sigma\left(T\big|_{R_{\{\lambda\}}}\right)=\{\lambda\}$, which implies $\sigma(S)=\{0\}$. Since $T$ is algebraic, the restriction $S$ is also algebraic. Let $m_{S}(z)$ be the minimal polynomial for $S$. Since $\sigma(S)=\{0\}$, so $m_{S}(z)=z^{m}$ for some integer $m \ge 1$. Thus there exists an integer $m \ge 1$ such that $S^{m}=0$. We will now prove that $S=0$. Suppose, in contrary, $S \ne 0$. Let $r$ be the smallest integer such that $S^{r}=0$; hence $r \ge 2$. Thus there exists $x_0 \in R_{\{\lambda\}}$ such that $S^{r-1}x_0 \ne 0$. Set $y:= S^{r-2}x_0$. Then
\[S^2y=S^{r}x_0=0, \quad \text{and}\;\;\; Sy=S^{r-1}x_0 \ne 0.\] Thus, $y \in \operatorname{Ker}(S^2)$ but $y \notin \operatorname{Ker}(S)$, which is a contradiction to the semisimplicity of $\lambda$. Therefore, $S=0$, which means 
\[T\big|_{R_{\{\lambda\}}}=\lambda I\big|_{R_{\{\lambda\}}}.\] Thus, every vector $x \in R_{\{\lambda\}}$ also belongs to the subspace $\operatorname{Ker}(T-\lambda I)$ and hence $R_{\{\lambda\}}\subseteq \operatorname{Ker}(T-\lambda I)$.

To prove the reverse inclusion, let $y \in \operatorname{Ker}(T-\lambda I)$, so that $Ty = \lambda y$. For any $z \in \rho(T)$ and any $x \in X$, we have
$(T - zI)^{-1}(y) = \dfrac{y}{\lambda - z}.$ Therefore,
\begin{align*}
P_{{\{\lambda\}}} y &= -\frac{1}{2\pi i} \int_{\Gamma} (T - zI)^{-1} (y) \, dz
= \frac{y}{2\pi i} \int_{\Gamma} \frac{1}{z - \lambda} \, dz = y,
\end{align*}
where $\Gamma$ is a positively oriented Cauchy contour enclosing $\lambda$ and avoiding $\sigma(T) \setminus {\{\lambda\}}$. Hence, $y \in \operatorname{Ran}(P_{{\{\lambda\}}}) = R_{\{\lambda\}}$, and we conclude that $\operatorname{Ker}(T-\lambda I) \subseteq R_{\{\lambda\}}$, completing the proof.

\end{enumerate}
\vspace{-0.3in}
\end{proof}
\begin{remark}\label{rem:npotenteigenprojection}
  Let $T$ be an $n$-potent operator on a Banach space $X$. Then for each $\lambda \in \sigma(T)$, the Riesz projection $P_{\{\lambda\}}$ coincides with the eigenprojection onto $\operatorname{Ker}(T-\lambda I)$.
\end{remark}

\begin{remark}
    Although the Riesz projections $P_{\{\lambda\}}$ is bounded, however, $\|P_{\{\lambda\}}\|$ can be made arbitrarily large. For $n \in \mathbb{N}$, we consider the following tri-potent operator $T_n$ on $X=\mathbb{C}^2$ with the $\ell^{\infty}$ norm given by the matrix:
    \[T_n=\begin{pmatrix}
        1 & -2n\\
        0 & -1
    \end{pmatrix}.\]
    One can verify that $T_n$ is tri-potent and $\sigma(T_n)=\{1,-1\}$. We compute the Riesz projection $P_{\{1\}}$ associated to the eigenvalue $1$:
    \[P_{\{1\}}=\begin{pmatrix}
        1 & -n\\
        0 & 0
    \end{pmatrix}.\] We see that $\|P_{\{1\}}\| \to \infty$ as $n \to \infty$ and hence the norm grows arbitrarily large.  
    
\end{remark}

\section{Spectral decomposition of $n$-potent operators and projections in $comb(T)$}\label{sec2:spec-decomposition} 
In this section we develop a spectral decomposition for $n$-potent operators and use it to analyze projections in the algebra $\operatorname{comb}(T)$. We use that the spectrum of $n$-potent operators is finite and is contained in $\mu_{n-1}\cup \{0\}$ and use it to derive explicit formulas for the Riesz projections associated with the nonzero eigenvalue of $T$. We also obtain the necessary and sufficient conditions for $T$ to be periodic.
\medskip

We start with a spectral decomposition of an $n$-potent operator $T$ on a Banach Space $X$. In the following theorem, we establish that an operator $T$ is $n$-potent if and only if it admits a spectral decomposition in terms of its Riesz projections. While this result resembles Corollary 3.3 of \cite{OJ} and Theorem 5.9-D-F in \cite{t}, it applies to a different class of operators. We will first recall a finite variant of resolution of identity. 

\begin{definition}\label{def:roi}
A finite family of projections $\{P_j\}_{j=1}^k \subseteq \mathcal{B}(X)$ is called a finite resolution of identity on $X$ if:
\begin{enumerate}
    \item $P_iP_j=0=P_jP_i$ for $i \ne j$ and $i, j \in \{1,2, \cdots, k\}$. 
    \item $I=\displaystyle\sum_{j=1}^{k}P_j$.
\end{enumerate}  
\end{definition}
The condition (1) is also referred to as the family $\{P_j\}_{j=1}^{k}$ being pairwise disjoint family. 
\smallskip

\noindent We are now ready to state the Theorem:
\begin{theorem}\label{resultonnpotentdecomposition}
Let $T$ be a bounded linear operator on a Banach space $X$. Then $T$ is an $n$-potent operator if and only if there exist scalars $\lambda_1, \dots, \lambda_k \in \{0, 1, \omega, \ldots, \omega^{n-2}\}$, where $\omega$ is a primitive $(n-1)^{\text{th}}$ root of unity, and unique projections $P_1, \dots, P_k$ $\in $ $\mathcal{B}(X)$ such that:
\begin{enumerate}[label=(\roman*)]
    \item Each non-zero $\lambda_j$ is distinct and satisfies $\lambda_j^n = \lambda_j$;
    \item $T = \displaystyle\sum_{j=1}^{k} \lambda_j P_j$;
    \item The family $\{P_j\}_{j=1}^k$ forms a finite resolution of identity.
\end{enumerate}

In particular, if one of the $ \lambda_j $'s is equal to $ 0 $, then the corresponding projection $ P_0 $ projects onto the null space $ \operatorname{Ker}(T) $.
\end{theorem}
\begin{proof}
Suppose there exist scalars $\lambda_1, \dots, \lambda_k \in \{0, 1,\omega,\cdots, \omega^{n-2}\}$ and projections $\{P_j\}_{j=1}^{k}$ satisfying properties $(i)$, and $(iii)$ such that
$T = \sum_{j=1}^{k} \lambda_j P_j.$
Without loss of generality, assume that $\lambda_1 = 0$. Then using the disjointness of $P_j$, we have:
\[
T^n = \sum_{j=1}^{k} \lambda_j^n P_j.
\]
Since each $\lambda_j$ satisfies $\lambda_j^n = \lambda_j$ (by condition (i)), we conclude:
\[
T^n = \sum_{j=1}^k \lambda_j^n P_j = \sum_{j=1}^k \lambda_j P_j = T,
\]
which shows that $T$ is $n$-potent.

To prove the converse, we assume that $T$ is an $n$-potent operator. Then by Proposition~\ref{SpecOfnPotentOpProp}, all eigenvalues of $T$ lie in $\{0, 1, \omega, \ldots, \omega^{n-2}\}$, and hence, each non zero eigenvalue $\lambda$ is distinct and satisfies $\lambda^n = \lambda$. Let $\{\lambda_1, \dots, \lambda_k\}$ denote the distinct eigenvalues of $T$, and for each $\lambda_j$, let $Q_j$ denote the corresponding Riesz projection. Using a similar technique used in Proposition ~\ref{basicpropsofeigenprojectionsprop}(v), we can show that the sum of the projections $\{Q_j\}_{j=1}^{k}$ is $I$. 

To show $Q_iQ_j=0$ for every pair $(i,j)$ with $i \neq j$, we first note that $\operatorname{Ran}(Q_i)$ and $\operatorname{Ran}(Q_j)$ are the eigenspaces corresponding to $\lambda_i$ and $\lambda_j$ respectively. Let us denote them by $E_{\lambda_i}$ and $E_{\lambda_j}$ respectively from Remark \ref{rem:npotenteigenprojection}. Since, $E_{\lambda_i} \cap E_{\lambda_j}=\{0\}$ and $Q_iQ_j=Q_jQ_i$ (by Proposition \ref{basicpropsofeigenprojectionsprop}(iii)), then for every $x \in X$, 
\[z:=Q_iQ_j(x)=Q_jQ_i(x)\] which implies $z \in \operatorname{Ran}(Q_i) \cap \operatorname{Ran}(Q_j)$ and hence $z=0$, satisfying property $(iii)$.

 To finish the proof, we will first show that $T=\sum_{j=1}^{k}\lambda_jQ_j$ where $\{Q_j\}_{j=1}^{k}$ are the Riesz projections corresponding to the eigenvalue $\{\lambda_j\}_{j=1}^{k}$. This can be shown as follows: Let $x \in X$. Since $\sum_{j=1}^{k}Q_j=I$, and $Q_j$ is an eigenprojection for all $1 \le j \le k$, then
 \[T(x)= \sum_{j=1}^{k}T(Q_{j}(x))= \sum_{j=1}^{k} \lambda_jQ_j(x).\] Since $x \in X$ is arbitrary, therefore, $T=\sum_{j=1}^{k}\lambda_jQ_j.$ 
 We will now prove the uniqueness of $Q_j$. Suppose that there exist projections $\{R_j\}_{j=1}^{k}$ and scalars $\lambda_1, \lambda_2, \dots, \lambda_k$ such that $T=\sum_{j=1}^{k}\lambda_jR_j$. We will show that for every $1 \leq j \leq k$, $Q_j=R_j$. It is enough to show that for every $1 \leq j \leq k$, $\operatorname{Ran}(R_j)=\operatorname{Ran}(Q_j)$ and $Ker(R_j)=Ker(Q_j)$. By the construction of $R_j$ and $Q_j$, it is clear that $\operatorname{Ran}(R_j)=\operatorname{Ker}(T-\lambda_jI)=\operatorname{Ran}(Q_j)$ for all $j$. To show $Ker(R_j)=Ker(Q_j)$ for every $1 \leq j \leq k$, we note that:
 $Ker(R_j)=\bigoplus_{\scriptstyle i = 1,\; i \ne j}^{k}\operatorname{Ran}(R_i)$. This is because of the following. Let $x \in Ker(R_j)$, which implies $R_j(x)=0$. Since, $\sum_{j=1}^{k}R_j=I$, hence,
\[x=\sum_{i=1, i \ne j}^{k}R_i(x) \in \bigoplus_{i=1, i\neq j}^{k}\operatorname{Ran}(R_i),\]
here the containment is direct sum since the decomposition is unique because of $(iii)$. To prove the converse, let $x \in \bigoplus_{i=1, i \ne j}^{k}\operatorname{Ran}(R_i)$. Since the projections $R_j$ satisfy $(iii)$, hence,
\[R_jx = \sum_{i=1, i \ne j}^{k}R_jR_i(x)=\sum_{i=1,i \ne j}^{k}R_i(R_j(x))=0,\] which implies $x \in Ker(R_j)$. Therefore, $Ker(R_j)=\bigoplus_{\scriptstyle i = 1,\; i \ne j}^{k}\operatorname{Ran}(R_i)$. Similarly, replacing $R_j$ by $Q_j$ we get $Ker(Q_j)=\bigoplus_{\scriptstyle i = 1,\; i \ne j}^{k}\operatorname{Ran}(Q_i)$. Therefore, by combining all the observations, we have the following:
\[Ker(R_j)= \bigoplus_{i=1, i\neq j}^{k}\operatorname{Ran}(R_i)= \bigoplus_{i=1, i\neq j}^{k}\operatorname{Ran}(Q_i)= Ker(Q_j).\] Hence, for all $1 \leq j \leq k$, $R_j=Q_j$ and that proves the Theorem. 

Assume $T$ is an $n$-potent operator and in particular, one of the $\lambda_j$'s is zero. Also let $P_0$ denote the projection corresponding to the eigenvalue $0$. We claim that $\operatorname{Ran}(P_0) = \operatorname{Ker}(T)$. 

First, let $x \in \operatorname{Ran}(P_0)$. Then $x = P_0 y$ for some $y \in X$, and since $T = \sum_j \lambda_j P_j$, it follows that
\[
T x = T(P_0 y) = \sum_{j=1}^{k} \lambda_j P_j(P_0 y) = \lambda_0 P_0 y = 0,
\]
as $\lambda_0 = 0$ and $P_j P_0 = 0$ for $j \ne 0$. Hence $x \in \operatorname{Ker}(T)$, and we conclude that $\operatorname{Ran}(P_0) \subseteq \operatorname{Ker}(T)$.

Conversely, let $x \in \operatorname{Ker}(T)$. Then
\[
0 = T x = \sum_{j=1}^{k} \lambda_j P_j x,
\]
which implies $\sum_{j \colon \lambda_j \neq 0} \lambda_j P_j x = 0$. Since the projections $\{P_j\}$ satisfy $P_i P_j = 0$ for $i \ne j$, it follows that $P_j x = 0$ for all $j$ with $\lambda_j \ne 0$. Hence,
\[
x = \sum_{j=1}^{k} P_j x = P_0 x,
\]
so $x \in \operatorname{Ran}(P_0)$. Thus, $\operatorname{Ker}(T) \subseteq \operatorname{Ran}(P_0)$. Therefore, we obtain $\operatorname{Ran}(P_0) = \operatorname{Ker}(T)$, as desired.
\end{proof}

\begin{remark}
Let $T$ be an $n$-potent operator on a Banach Space $X$ and define $Y:=\operatorname{Ran}(T^{n-1})$, the subspace corresponding to the nonzero eigenvalues of $T$. From \autoref{resultonnpotentdecomposition} we get that corresponding to any non-zero value $\{\lambda_1, \lambda_2,\cdots, \lambda_r\}$ with $1 \le r\le n-1$ in the spectrum of $T$, there exist Riesz projections $\{P_{\lambda_1}, P_{\lambda_2}, \cdots, P_{\lambda_r}\}$ such that:
    \[\{P_{\lambda_j}\}_{j=1}^{r}\;\; \text{forms a resolution of identity, and}\; Y=\bigoplus_{j=1}^{r}\operatorname{Ran}(P_{\lambda_j}),\] and $T$ has the following decomposition on $Y$: 
    \[T=\lambda_1P_{\lambda_1}+\lambda_2P_{\lambda_2}+\cdots+\lambda_{r}P_{\lambda_{r}}.\]
This decomposition coincides with the decomposition that appears in the definition of generalized circular projections (see, e.g., Definition 1.1 in \cite{ILP}). The only condition we need here is that the operator $T|_{Y}$ is an isometry for the given norm. Although, in general, an $n$-potent operator need not be an isometry, we may equip $Y$ with an equivalent norm such as:
    \[\|y\|_{Y}:= \max_{0 \le k \le n-2}\|T^ky\|,\]
    
to make $T$ an isometry on $Y$. Thus with respect to the renorming, the family $\{P_{\lambda_1},P_{\lambda_2}, \cdots, P_{\lambda_r}\}$ in the decomposition of $T$ forms a family of generalized circular projections. 
    
\end{remark}

The following Theorem shows a correspondence between each non-zero element of the spectrum an $n$-potent operator and its associated Riesz projection. In particular, we can compute the eigen (Riesz) projection corresponding to each non-zero $\lambda \in \sigma(T).$ We will start with a lemma which will be useful for the rest of the paper. 
\begin{lemma}\label{computation-of-inverse}
    Let $n \ge 2$ be an integer, and let $T \in \mathcal{B}(X)$ be an $n$-potent operator. Then for every $z \in \rho(T)$, we have the following expansion:
    \[(zI-T)^{-1}= \left[\frac{I}{z} + \bigg(\frac{1}{z(z^{n-1}-1)}\bigg)T^{n-1}+\sum_{k=1}^{n-2} \left( \frac{z^{n-k-2}}{z^{n-1} - 1} \right) T^k\right].\]
\end{lemma}
\begin{proof}
    Define:  
    \[R(z):=\frac{I}{z} + \bigg(\frac{1}{z(z^{n-1}-1)}\bigg)T^{n-1}+\sum_{k=1}^{n-2} \left( \frac{z^{n-k-2}}{z^{n-1} - 1} \right) T^k,\] and it is enough to prove that $(zI-T)R(z)=I$ for all $z \in \rho(T)$. To show this we expand and use the $n$-potency of $T$: 
\begin{align*}
    (zI-T)R(z)&= (zI-T)\bigg[\frac{I}{z} + \bigg(\frac{1}{z(z^{n-1}-1)}\bigg)T^{n-1}+\sum_{k=1}^{n-2} \left( \frac{z^{n-k-2}}{z^{n-1} - 1} \right) T^k\bigg]\\
&= \bigg(I - \frac{T}{z}\bigg)
   + \bigg(\frac{T^{n-1}}{z^{n-1}-1} - \frac{T}{z(z^{n-1}-1)}\bigg)
   + \sum_{k=1}^{n-2} \frac{z^{n-k-1}T^{k} - z^{n-k-2}T^{k+1}}{z^{n-1}-1}. \tag{1}
\end{align*}
We simplify all the denominators, evaluate the telescopic sum and observe that: 
\begin{align*}
    \sum_{k=1}^{n-2} \frac{z^{n-k}T^{k} - z^{n-k-1}T^{k+1}}{z(z^{n-1}-1)}&=\frac{\left(z^{n-1}T-z^{n-2}T^2\right)+\left(z^{n-2}T^2-z^{n-3}T^3\right)+\cdots+\left(z^{2}T^{n-2}-zT^{n-1}\right)}{z(z^{n-1}-1)}\\
    &=\frac{z^{n-1}T-zT^{n-1}}{{z(z^{n-1}-1)}}.
\end{align*}
We simplify the expression in (1): 
\[(zI-T)R(z)= \frac{(z^n-z)I}{z(z^{n-1}-1)}=I.\]
Since for $z \in \rho(T)$, $(zI-T)$ in invertible, hence the lemma is proved.
\end{proof}
We are now ready to state and prove the Theorem.
\begin{theorem}\label{thm:spectrumandeigenprojection}
    Let $n\geq 2$ be a positive integer, and $T$ be an $n$-potent operator on a Banach Space $X$. Let $\omega= e^{\frac{2\pi i}{n-1}}$ denote the primitive $(n-1)$-th root of unity. Then, for each $j=0,1, \cdots, n-2$, the following are equivalent: 
    \begin{enumerate}[label=(\roman*)]
        \item $\omega^{j} \in \sigma(T),$
        \item The operator \[P_{\{\omega^{j}\}}= \frac{1}{n-1}\sum_{k=1}^{n-1}\omega^{-jk}T^k\] is a non-zero projection that satisfies 
        \[TP_{\{\omega^{j}\}}=P_{\{\omega^{j}\}}T, \qquad TP_{\{\omega^{j}\}}=\omega^{j}P_{\{\omega^{j}\}}.\]
    \end{enumerate}
    In particular, the projections $P_{\{\omega^{j}\}}$ coincides with the Riesz projection of $T$ associated with the eigenvalue $\omega^{j}$.
\end{theorem}
\begin{proof}
    Since $T$ is an $n$-potent operator, by Proposition \ref{SpecOfnPotentOpProp}, we get $\sigma(T)\subseteq\{0, 1, \omega, \dots, \omega^{n-2}\}$. We will prove the equivalence by proving $(i)$ implies $(ii)$ first. Let us assume that $\omega^{j} \in \sigma(T)$ for some $j \in \{0,1,\cdots, n-2\}$. Let $\Gamma_{j}$ be a Cauchy contour in the resolvent set of $T$ around only $\omega^j$. Then the corresponding Riesz projection is given by (Definition \ref{DefOfEigenProj}),
    \begin{equation*}
       P_{\{\omega^{j}\}}:= -\frac{1}{2\pi i}\int_{\Gamma_{j}}(T-zI)^{-1}\, dz. \nonumber
    \end{equation*} 
    Since the contour $\Gamma_j$ does not contain $0$, then for some $0 \leq j \leq (n-2)$, using a similar expansion in Lemma \ref{computation-of-inverse} we have
\begin{align}\label{rieszp}
    P_{\{\omega^j\}} &= \frac{1}{2\pi i} \int_{\Gamma_j} \left[\frac{I}{z}+ \sum_{k=1}^{n-1}\left(\frac{z^{n-k-2}}{z^{n-1}-1}\right)T^{k}\right] \, dz \nonumber\\
    &= \frac{1}{2\pi i} \int_{\Gamma_j}\frac{1}{z}\, dz 
    + \sum_{k=1}^{n-1}\frac{1}{2\pi i}\int_{\Gamma_j}\left(\frac{z^{n-k-2}}{z^{n-1}-1}\right)T^{k} \, dz \nonumber\\
    &= \sum_{k=1}^{n-1}\frac{1}{2\pi i}\int_{\Gamma_j}\left(\frac{z^{n-k-2}}{z^{n-1}-1}\right)T^{k} \, dz
\end{align}
We denote $c_k^{(j)}:= \dfrac{1}{2\pi i}\displaystyle\int_{\Gamma_j}\frac{z^{n-k-2}}{z^{n-1}-1}\,dz.$ Using the residue theorem, after evaluating the integral we get, 
\[c_k^{(j)}= \frac{\omega^{j(n-k-2)}}{\prod_{m=0, m \ne j}^{n-2}(\omega^{j}-\omega^{m})}\]
Then by equation (\ref{rieszp}), we have, $P_{\{\omega^{j}\}}= \displaystyle\sum_{k=1}^{n-1}c^{(j)}_{k}T^{k}$. We can also simplify the expression of $c_k^{(j)}$ as follows:
\begin{align*}
    c_k^{(j)}= \frac{\omega^{j(n-k-2)}}{\prod_{m=0, m \ne j}^{n-2}(\omega^{j}-\omega^{m})}= \frac{\omega^{j(n-2)}\omega^{-jk}}{\prod_{m=0, m \ne j}^{n-2}(\omega^{j}-\omega^{m})}:= C_j\omega^{-jk},
\end{align*}
where $C_j:= \dfrac{\omega^{j(n-2)}}{\prod_{m=0, m \ne j}^{n-2}(\omega^{j}-\omega^{m})}$. To finish the proof, we need to show that $C_j=\dfrac{1}{n-1}.$ We note that,
\begin{align*}
C_j = \dfrac{\omega^{j(n-2)}}{\omega^{j(n-2)}\prod_{m=0, m \ne j}^{n-2}(1-\omega^{m-j})}  &= \frac{1}{\prod_{m=0, m \ne j}^{n-2}(1-\omega^{m-j})}\\ &= \frac{1}{\prod_{s=1}^{n-2}(1-\omega^{s})}
    \end{align*}
Since $\omega$ is a primitive $(n-1)$-th root of unity, then $\prod_{s=1}^{n-2}(1-\omega^{s})=(n-1).$ Therefore, $P_{\{\omega^{j}\}}= \dfrac{1}{n-1}\displaystyle\sum_{k=1}^{n-1}\omega^{-jk}T^k$. To show $TP_{\{\omega^{j}\}}=\omega^{j}P_{\{\omega^{j}\}}$ we use the $n$-potency of $T$ and compute:
\begin{align*}
TP_{\{\omega^{j}\}} = \frac{1}{n-1}\sum_{k=1}^{n-1}\omega^{-jk}T^{k+1} &= \frac{1}{n-1}\sum_{r=2}^{n-1}\omega^{-j(r-1)}T^{r}\nonumber\\&= \omega^{j}\left(\frac{1}{n-1}\sum_{r=2}^{n-1}\omega^{-jr}T^{r}\right)
\end{align*}
We rewrite the equation as:
\begin{align*}
   TP_{{\{\omega^{j}\}}}= \frac{\omega^{j}}{n-1}\left(\sum_{k=1}^{n-1}\omega^{-jk}T^{k}-\omega^{-j}T+\omega^{-jn}T^n\right)&=\frac{\omega^{j}}{n-1}\left(\sum_{k=1}^{n-1}\omega^{-jk}T^{k}\right)\\&=\omega^{j}P_{{\{\omega^{j}\}}}, 
\end{align*}
hence proved. The other computation is similar. 

To prove the converse direction, we assume that the operator $P_{\{\omega^{j}\}}= \frac{1}{n-1}\sum_{k=1}^{n-1}\omega^{-jk}T^k$ is non zero projection satisfying conditions mentioned in the hypothesis. Since the spectrum of an $n$-potent contains eigenvalues only, so it is enough to show that $\omega^{j} \in \sigma_p(T)$. Pick any non zero $x \in \operatorname{Ran}(P_{\{\omega^j\}})$. Then $P_{\{\omega^j\}}(x)=x$ and 
\[Tx= TP_{\{\omega^j\}}(x)=\omega^{j}P_{\{\omega^j\}}(x)=\omega^{j}x,\]
which proves $\omega^{j} \in \sigma(T)$ and thereby proving the Theorem. 
\end{proof}

\begin{remark}
    If $T$ is an $n$-potent operator, from the spectral representation in \autoref{resultonnpotentdecomposition} and the form of eigenprojections $P_{\{\omega^{j}\}}$ from \autoref{thm:spectrumandeigenprojection}, we get 
    \[T= \sum_{j=0}^{n-2}\omega^{j}P_{\{\omega^j\}}.\] Since $\omega$ is a  primitive $(n-1)^{\text{th}}$ root of unity, we can rewrite the decomposition as:
    \begin{equation}\label{eqdecomp}
        T = \frac{1}{n-1}\left[\sum_{j=0}^{n-2}\omega^{j}\left((n-1)P_{\{\omega^{j}\}}+I\right)\right].
    \end{equation} In particular, for $n=3$, where the non-zero spectrum consists of $\{-1,1\}$, the above representation gives 
    \[T=\frac{1}{2}\left[(2P_{\{1\}}+I)-(2P_{\{-1\}}+I)\right],\] which coincides with Corollary 1 of \cite{BY}. Therefore, the representation of $T$ in \eqref{eqdecomp} generalizes the representation obtained in the tri-potent case in \cite{BY}.
\end{remark}

The next Theorem characterizes the connections of an $n$-potent $T$ with its eigenprojections. We recall that a linear operator $T$ on a Banach Space $X$ is said to be periodic if there exists $m \in \mathbb{N}$ such that $T^{m}=I$ and $T^{k} \ne I$ for all $1 \le k <m$; in that case we say $T$ has period $m$. 
We will now state and prove the Theorem. 
\begin{theorem}\label{MainThm1}
Let $n \geq 2$ be an integer, and let $T \in \mathcal{B}(X)$ be an $n$-potent operator on $X$. Then the following statements hold:
\begin{enumerate}[label=(\roman*)]
\item $T^{n-1}$ is a projection. Moreover, $T^{n-1}$ is the Riesz projection associated with the nonzero eigenvalues of $T$. 
\item  $0\in\sigma (T) $ if and only if  $I-T^{n-1}$ is the nonzero Riesz projection associated with the eigenvalue $0$. In this case, 
\[\operatorname{Ran}(I-T^{n-1})=\operatorname{Ker}(T^{n-1})=\operatorname{Ker}(T).\]
\item $T$ is periodic if and only if $0 \notin \sigma(T).$ Moreover,  in this case, $T$ is invertible. 

\end{enumerate}
\end{theorem}

\begin{proof}
\begin{enumerate}[label=(\roman*)]
\item  Since $T$ is $n$-potent, by \autoref{resultonnpotentdecomposition}, there exist scalars $\lambda_1, \lambda_2, \cdots, \lambda_k \in \{0, 1, \omega, \ldots, \omega^{n-2}\}$, and a family of projections $\{P_j\}_{j=1}^{k}$ such that:
\[
T = \sum_{\lambda_j \ne 0} \lambda_j P_j.
\]

Then, using the fact that for each $j=1,2, \cdots, k$, $P_{j}$ is a projection, and each non zero $\lambda_j$ has order $(n-1)$ we get:
\[
T^{n-1} = \left( \sum_{\lambda_j \ne 0} \lambda_j P_j \right)^{n-1} = \sum_{\lambda_j \ne 0} (\lambda_j)^{n-1} P_j = \sum_{\lambda_j \ne 0} P_j.
\]

Since the projections $\{P_j\}$ satisfy conditions of Definition \ref{def:roi}, we get: 
\[
\left(T^{n-1}\right)^2 = \left( \sum_{\lambda_j \ne 0} P_j \right)^2 = \sum_{\lambda_j \ne 0} P_j,
\]
Hence, $(T^{n-1})^2 = T^{n-1}$. Therefore, $T^{n-1}$ is a projection. 
\medskip

We will now show that $T^{n-1}$ is the Riesz projection associated with the nonzero eigenvalues of $T$. Let $P_{\sigma(T)\setminus \{0\}}$ denote the Riesz projection associated with the set of all non zero eigenvalues. For our notational convenience we will denote elements of $\sigma(T)\setminus \{0\}$ as $\{\omega^{j}: j=0,1,\cdots, n-2\}$, where $\omega^{n-1}=1$. Note that
\[\sum_{j=0}^{n-2}\omega^{-jk}=\begin{cases}
0,     & \text{if } 1 \le k\le n-2,\\
n-1        & \text{if } k = n-1.
\end{cases}\]

Then using Proposition~\ref{basicpropsofeigenprojectionsprop}(v) and \autoref{thm:spectrumandeigenprojection}, we have:
\begin{align*}
    P_{\sigma(T)\setminus \{0\}} &= \sum_{j=0}^{n-2}P_{\{\omega^{j}\}} \\
    &=\sum_{j=0}^{n-2}\left(\sum_{k=1}^{n-1}\frac{1}{n-1}\omega^{-jk}\right)T^{k}\\&= \frac{1}{n-1}\sum_{k=1}^{n-1}T^{k}\left(\sum_{j=0}^{n-2}\omega^{-jk}\right)
\end{align*}

All the terms in the above sum vanish except $k=n-1$. Therefore, $$P_{\sigma(T)\setminus \{0\}}= \frac{1}{n-1}\left[(n-1)T^{n-1}\right]=T^{n-1},$$ which is what we wanted to prove. 

\item Since $T$ is an $n$-potent operator. By Proposition~\ref{SpecOfnPotentOpProp}, we know that the spectrum of $T$ is a subset of $\{0, 1, \omega, \dots, \omega^{n-2}\}$,
where $\omega = e^{2\pi i / (n-1)}$ is a primitive $(n-1)$-th root of unity.

To prove the claim, first assume that $0 \in \sigma(T)$. Using Lemma~\ref{computation-of-inverse}, we obtain:
\begin{equation}\label{eq:riesz-projection-eq-1}
P_{\{0\}} = \frac{1}{2\pi i} \int_{\Gamma_{0}} \left[ \frac{I}{z} + \bigg(\frac{1}{z(z^{n-1}-1)}\bigg)T^{n-1}+\sum_{k=1}^{n-2} \left( \frac{z^{n-k-2}}{z^{n-1} - 1} \right) T^k \right] \, dz,
\end{equation}
where $\Gamma_0$ is a Cauchy contour that encloses $0$ in its interior. Then a direct computation reveals that the Riesz projection $P_{\{0\}}$ associated with the eigenvalue $0$ is given by
$P_0 = I - T^{n-1}$. Conversely, suppose that $I - T^{n-1}$ is a nonzero projection and that $0 \notin \sigma(T)$. Then $T$ is invertible, and hence $T^{n-1}=I$. But that is a contradiction to the our assumption. Hence, it follows that $0 \in \sigma(T)$. To verify the identity $\operatorname{Ran}(I - T^{n-1}) = \operatorname{Ker}(T^{n-1})$, let $y \in \operatorname{Ran}(I - T^{n-1})$. Then there exists $x \in X$ such that
\[
y = (I - T^{n-1})(x) = x - T^{n-1}(x).
\]
Applying $T^{n-1}$ to both sides yields and using the $n$-potency of $T$:
\[
T^{n-1}(y) = T^{n-1}(x - T^{n-1}(x)) = T^{n-1}(x) - T^{2n-2}(x) = T^{n-1}(x) - T^{n-1}(x) = 0,
\]
so $y \in \operatorname{Ker}(T^{n-1})$. For the reverse direction, suppose $y \in \operatorname{Ker}(T^{n-1})$. Then $T^{n-1}(y) = 0$, so $(I - T^{n-1})(y) = y$,
i.e., $y \in \operatorname{Ran}(I - T^{n-1})$. Thus,
$\operatorname{Ran}(I - T^{n-1}) = \operatorname{Ker}(T^{n-1})$. Similarly, using the $n$-potency of $T$, we can show that $\operatorname{Ker}(T^{n-1})=\operatorname{Ker}(T)$,
and hence $\operatorname{Ran}(I-T^{n-1})=\operatorname{Ker}(T^{n-1})=\operatorname{Ker}(T)$.

\item Suppose $T$ is a periodic operator. Then there exists $m \in \mathbb{N}$ such that $T^{m}=I$. Hence $T$ is invertible and so $0 \notin \sigma(T)$. Conversely if $0 \notin \sigma(T)$, then $T$ is invertible and since $T$ is an $n$-potent operator, using the invertibility of $T$ we get $T^{n-1}=I$, which implies the periodicity of $T$.
\end{enumerate}
\end{proof}

\begin{remark}
    In general, if $T$ is an $n$-potent operator, then the eigenprojections may not be Hermitian. For example, consider the Hilbert Space $X=\mathbb{C}^2$ with the Euclidean norm and let $T=\begin{pmatrix}
        1 & 1\\
        0 & 0
    \end{pmatrix}$. It is clear that $T$ is a projection, and hence a $2$-potent operator and $\sigma(T)=\{0,1\}$. The Riesz projection associated with $\lambda=0$ is given by:
    \[P_0=I-T=\begin{pmatrix}
        0 & -1\\
        0 & 1
    \end{pmatrix}.\]
    
    We will show that $P_0$ is not Hermitian. By Lemma~2.1 in \cite{J}, a projection
$P$ is Hermitian if and only if it is bicircular; that is, if and only if
$P+\alpha(I-P)$ is an isometry for every unimodular
$\alpha$. Let $|\alpha|=1$, $\alpha \neq 1$ and we compute:
    \[A_{\alpha}=P_0+\alpha(I-P_0)=\begin{pmatrix}
        \alpha & \alpha-1\\
        0 & 1
    \end{pmatrix}.\]
    Since $X$ is a Hilbert Space, so if $A_{\alpha}$ were an isometry, then $A_{\alpha}^{\ast}A_{\alpha}=I$. But
    \[A_{\alpha}^{\ast}A_{\alpha} = \begin{pmatrix}
        \overline{\alpha} & 0\\
        \overline{\alpha}-1 & 1
    \end{pmatrix} \begin{pmatrix}
        \overline{\alpha} & \alpha -1\\
        0 & 1
    \end{pmatrix}= \begin{pmatrix}
        1 & 1-\overline{\alpha}\\
        1-\alpha & (\alpha-1)(\overline{\alpha}-1)+1
    \end{pmatrix},\]which implies $A_{\alpha}^{\ast}A_{\alpha}=I$ if and only if $\alpha=1$. Therefore, the operator $P_0+\alpha(I-P_0)$ is not an isometry for any $\alpha \ne 1$ and hence the eigenprojection $P_0$ is not Hermitian. 
\end{remark}
For an $n$-potent operator $T \in \mathcal{B}(X)$, we recall the algebra generated by the powers of $T$: \[\operatorname{comb}(T):= \left\{\sum_{i=1}^{k}a_iT^{i}\colon a_i\in \mathbb{C}, 1 \leq k \leq n-1\right\}.\]
Since $T^n=T$, every polynomial in $T$ reduces to a linear combination of $\{T,T^2, \dots, T^{n-1}\}$ and hence $\operatorname{comb}(T)$ is a finite dimensional subalgebra of $\mathcal{B}(X)$. 
In this Section we study the structure of projections that belong to $\operatorname{comb}(T)$. Our main result shows that such a projection can be written uniquely as a linear combination of the Riesz projections associated with the nonzero eigenvalues of $T$.  We will now state and prove the Theorem. 
\begin{theorem}\label{thm:proj-conv-comb-result}
Let $n \geq 2$ and $T \in \mathcal{B}(X)$ be an $n$-potent operator on a Banach Space $X$. If $P \in comb(T)$ is a projection, then $P$ admits a unique spectral decomposition:
\[
P = \sum_{\lambda \in \sigma(T)} \beta_\lambda P_{\{\lambda\}}, \quad \text{with } \beta_\lambda = \sum_{i=1}^{k} a_i \lambda^i,
\]
where $P_{\{\lambda\}}$ denotes the Riesz projection of $T$ associated with the eigenvalue $\lambda$. Furthermore, $\beta_{\lambda} \in \{0,1\}$ for each $\lambda \in \sigma(T)$, and for $\beta_\lambda \ne 0$, $\operatorname{Ran}(P)$ is the direct sum of subspaces $P_{\{\lambda\}}(X)$.
\end{theorem}
\begin{proof}
Since $T$ is an $n$-potent operator, by \autoref{resultonnpotentdecomposition}, we have a spectral decomposition of $T$ of the form:
\[
T = \sum_{\lambda \in \sigma(T)} \lambda\, P_{\{\lambda\}},
\]
where $\sigma(T) \subseteq \{0, 1, \omega, \dots, \omega^{n-2}\}$, $\omega = e^{\frac{2\pi i}{n-1}}$, and the Riesz projections $P_{\{\lambda\}}$ form a resolution of identity. 

Now let $P = \sum_{i=1}^{k} a_i T^i \in \operatorname{comb}(T)$. Then, for each $i\ge 1$ we have:
\[
T^i = \left( \sum_{\lambda \in \sigma(T)} \lambda P_{\{\lambda\}} \right)^i
= \sum_{\lambda \in \sigma(T)} \lambda^i P_{\{\lambda\}},
\]
since $P_{\{\lambda\}}P_{\{\mu\}}=0$ whenever $\lambda \neq \mu$ and each $P_{\{\lambda\}}$ is a projection. 
Thus,
\[
P = \sum_{i=1}^{k} a_i T^i
= \sum_{i=1}^{k} a_i \sum_{\lambda \in \sigma(T)} \lambda^i P_{\{\lambda\}}
= \sum_{\lambda \in \sigma(T)} \left( \sum_{i=1}^{k} a_i \lambda^i \right) P_{\{\lambda\}}.
\]

Define $\beta_\lambda := \sum_{i=1}^{k} a_i \lambda^i \in \mathbb{C}$. Then,
\[
P = \sum_{\lambda \in \sigma(T)} \beta_\lambda P_{\{\lambda\}}.
\]
Using the fact that $P$ is a projection, we get
\[P^2 = \left(\sum_{\lambda}\beta_{\lambda}P_{\{\lambda\}}\right)^2=\sum_{\lambda}\beta_{\lambda}^2P_{\{\lambda\}},\] and hence 
\[\sum_{\lambda}\beta_{\lambda}^2P_{\{\lambda\}}=\sum_{\lambda}\beta_{\lambda}P_{\{\lambda\}}.\]
Thus, $\sum_{\lambda}(\beta_{\lambda}^2-\beta_{\lambda})P_{\{\lambda\}}=0.$ Let us fix $\mu \in \sigma(T)$ and $x \in \operatorname{Ran}(P_{\{\mu\}})$. Hence $P_{\{\mu\}}(x)=x$ and for $\lambda \ne \mu$, $P_{\{\lambda\}}(x)=0$. Hence $(\beta_{\mu}^2-\beta_{\mu})x=0$ for all $x \in \operatorname{Ran}(P_{\{\mu\}})$, which implies $\beta_{\mu}^2-\beta_{\mu}=0$ and so $\beta_{\mu} \in \{0,1\}$ for all $\mu \in \sigma(T)$. The uniqueness follows similarly. 
\medskip

To finish the proof, first note that since $X=\displaystyle\bigoplus_{\lambda \in \sigma(T)}P_{\{\lambda\}}(X)$, then 
\[Px=\sum_{\lambda}\beta_{\lambda}P_{\{\lambda\}}(x)=\sum_{\lambda}\beta_{\lambda}x_{\lambda},\] where $x=\sum_{\lambda}x_{\lambda}$ with $x_{\lambda}\in P_{\{\lambda\}}(X)$. If $\beta_{\lambda}\ne 0$ then it follows that 
\[
\operatorname{Ran}(P)
= \bigoplus\limits_{\substack{\lambda \in \sigma(T)\\ \beta_\lambda \ne 0}}
  P_{\{\lambda\}}(X),\]which is what we wanted to prove. 
\end{proof}
An immediate consequence of \autoref{thm:proj-conv-comb-result} is that the projections in $\operatorname{comb}(T)$ have a rich algebraic structure:
\begin{corollary}\label{boolean-algebra-isomorphism}
Let $T$ be an $n$-potent operator and let $\sigma_{0}(T)=\sigma(T)\setminus \{0\}$. Then the following statements hold:
\begin{enumerate}[label=(\roman*)]
    \item Each projection $P \in \operatorname{comb}(T)$ has the form
    \[P=\sum_{\lambda \in S}P_{\{\lambda\}},\]
    for some $S \subseteq \sigma_{0}(T)$.
    \item The map $S \mapsto P_{S}:= \displaystyle\sum_{\lambda \in S}P_{\{\lambda\}}$ is a bijection from $\mathcal{P}(\sigma_{0}(T))$ onto the set of all projections in $\operatorname{comb}(T)$.
    \item For $S,R \subseteq \sigma_{0}(T)$, we have 
    \[P_{S}P_{R}=P_{S \cap R}, \quad P_{S}+P_{R}-P_{S}P_{R}=P_{S\cup R}, \quad I-P_{S}=P_{\sigma_{0}(T)\setminus S}.\]
\end{enumerate}
It follows that the family $\{P_{S} : S \subseteq \sigma_{0}(T)\}$ forms a Boolean algebra of projections under the operations
\[
    P_S \vee P_R := P_{S \cup R}, \qquad
    P_S \wedge P_R := P_{S \cap R}, \qquad P_{S^{c}} = P_{\sigma_0(T) \setminus S}
\] which is isomorphic to the Boolean algebra $(\mathcal{P}(\sigma_0(T)), \cup, \cap,^{c})$.
\end{corollary}
\begin{proof}
    (i) By \autoref{thm:proj-conv-comb-result}, every projection $P \in \operatorname{comb}(T)$ has a unique representation
    \[P=\sum_{\lambda \in \sigma(T)}\beta_{\lambda}P_{\{\lambda\}}, \quad \text{with}\;\; \beta_\lambda\in\{0,1\}.\]
    Let $S:=\{\lambda \in \sigma_{0}(T):\beta_{\lambda}=1\}$. Note that since $\beta_0=0$, hence,
    \[P=\sum_{\lambda \in S}P_{\{\lambda\}},\] and we are done. 

    (ii) To prove injectivity, suppose $P_{S}=P_{R}$ for some $S,R \subseteq \sigma_{0}(T)$. Then
    \begin{equation}\label{injectivity-of-P_R}
    \sum_{\lambda \in S}P_{\{\lambda\}}=\sum_{\mu \in R}P_{\{\mu\}}.
    \end{equation}
Fix $\lambda_0 \in S$. Multiplying by $P_{\{\lambda_0\}}$ on the left to the equation~(\ref{injectivity-of-P_R}) and using the pairwise disjointness, we obtain:
\begin{align*}
P_{\{\lambda_0\}} P_S
&= P_{\{\lambda_0\}} P_R,\\
P_{\{\lambda_0\}} \sum_{\lambda \in S} P_{\{\lambda\}}
&= P_{\{\lambda_0\}} \sum_{\mu \in R} P_{\{\mu\}},\\
\sum_{\lambda \in S} P_{\{\lambda_0\}} P_{\{\lambda\}}
&= \sum_{\mu \in R} P_{\{\lambda_0\}} P_{\{\mu\}}.
\end{align*}
Since $\lambda_0 \in S$, the left-side of the above sum is $P_{\{\lambda_0\}}$ and the only way it is equal to the right side is if $\lambda_0 \in R$. Hence $S \subseteq R$. Similarly we can show that if $\lambda_{0} \notin S$ then $\lambda_{0} \notin R$ as well. Therefore $S=R$ and the map is injective. Surjectivity follows from part $(i)$. 

(iii) To prove this we compute:
\[
P_S P_R=\sum_{\lambda\in S}\sum_{\mu\in R} P_{\{\lambda\}}P_{\{\mu\}}=\sum_{\lambda\in S\cap R} P_{\{\lambda\}}=P_{S\cap R}.
\]
Similarly, 
\[P_{S}+P_{R}-P_{S}P_{R}= \sum_{\lambda \in S}P_{\{\lambda\}}+\sum_{\mu \in R}P_{\{\mu\}}-\sum_{\nu \in S\cap R}P_{\{\nu\}}=\sum_{\lambda \in S \cup R}P_{\{\lambda\}}=P_{S \cup R},\]
and 
\[I-P_{S}=P_0+\sum_{\lambda \in \sigma_{0}(T)}P_{\{\lambda\}}-\sum_{\mu \in S}P_{\{\mu\}}=P_{\sigma_0(T)\setminus S}.\]

Since the map $S \mapsto P_{S}$ is a bijection and it preserves the operations, hence it is an isomorphism of Boolean algebras. 
\end{proof}

\section{Examples and cases of lower orders}\label{sec3-examples}
In this Section, we discuss a few Examples which illustrate how the general theory of $n$-potent operators and their spectral decompositions is applied both in classical Banach spaces and low-order finite dimensional cases. We also classify all the projections in $\operatorname{comb}(T)$ for a $5$-potent operator $T$ and illustrate \autoref{thm:proj-conv-comb-result} to find a decomposition of some selected projections $P$. We start with an example of a $(2k+1)$-potent operator on $C[0,1]$. 

\begin{example}\label{mainexample}
Let $X = C[0,1]$, and fix a positive integer $k$. Define an operator $T \in \mathcal{B}(X)$ by
\[
Tf(t) = \frac{1}{2} e^{\frac{\pi i}{k}} \big(f(1 - t) - f(t)\big), \quad \text{for all } f \in X.
\]
Then $T$ is $(2k+1)$-potent, and can be seen as follows. Let $R \in \mathcal{B}(X)$ be the reflection operator given by:
\[Rf(t):= f(1-t),\quad t \in [0,1].\] Setting $a=\frac{1}{2}e^{\frac{i\pi}{k}}$, we write $T=a(R-I)$. A direct computation shows that for $n \ge 2,$
\[T^{n}= (-2a)^{n-1}T. \]Hence $T^{2k+1}=(-2a)^{2k}T=T$ and so $T$ is $(2k+1)$-potent. We claim that 
\[
\sigma(T) = \{0, \omega^{k+1}\},
\]
where $\omega = e^{\frac{2\pi i}{2k}}$ is a primitive $(2k)$-th root of unity. Note that constant functions satisfy $Tf = 0$, so $\operatorname{Ker}(T) \neq \{0\}$, and hence $0 \in \sigma(T)$. Also, we note that $Tf=0$ if and only if $f(t)=f(1-t)$ for all $t \in [0,1]$. So the eigenspace corresponding to the eigenvalue $0$ is $\operatorname{Ker}(T)=\{f \in C[0,1]: f(1-t)=f(t)\}$.
By part $(ii)$ of \autoref{MainThm1}, $P_{\{0\}}:=I-T^{2k}$ is the Riesz projection and in fact we will show that it is exactly the eigenprojection associated with the eigenvalue $0$.

\vspace{1mm} 
If $ \lambda \in \mathbb{C} $ and $ f \neq 0 $ satisfies $ Tf = \lambda f $, then
\begin{equation}\tag{2}\label{eq:eig-0}
   f(1 - t) = (2\lambda e^{-\frac{\pi i}{k}} + 1) f(t) 
\end{equation}
In particular, taking $\lambda=0$, we see that any non-zero continuous function satisfying $f(t)=f(1-t)$ is an eigenfunction associated to the eigenvalue $0$. We replace $t$ by $1-t$ in equation (\ref{eq:eig-0}) and get 
\[f(t)=(2\lambda e^{-\frac{\pi i}{k}} + 1)f(1-t)=(2\lambda e^{-\frac{\pi i}{k}} + 1)^2f(t).\] Choose a point $t_0 \in [0,1]$ such that $f(t_0)\ne 0$. Hence, $(2\lambda e^{-\frac{\pi i}{k}} + 1)^2=1$, which implies $2\lambda e^{-\frac{\pi i}{k}} + 1=\pm 1$. If $2\lambda e^{-\frac{\pi i}{k}} + 1=1$ then $\lambda=0$ and if $2\lambda e^{-\frac{\pi i}{k}} + 1=-1$, then $\lambda= -e^{\frac{i\pi}{k}}=\omega^{k+1}.$ If $\lambda= -e^{\frac{i\pi}{k}}$, then from equation (\ref{eq:eig-0}), the eigenspace associated to $\lambda=-e^{\frac{i\pi}{k}}$ is the set $\{f \in C[0,1]: f(1-t)=-f(t)\}$. Therefore, $\{0,\omega^{k+1}\} \subseteq \sigma(T).$ It remains to show that for $\lambda \notin \{0,\omega^{k+1}\}$, the operator $T-\lambda I$ is invertible. For $ \lambda \notin \{0,\omega^{k+1}\} $, using Lemma~\ref{computation-of-inverse}, we have
\[(T - \lambda I)^{-1} = -\left[\frac{I}{\lambda}+\left(\frac{1}{\lambda(\lambda^{2k}-1)}\right)T^{2k} + \sum_{j=1}^{2k-1} \frac{\lambda^{2k - j-1}}{(\lambda^{2k-1} - 1)} T^j \right],\]
which shows $(T-\lambda I)$ is a bounded, and invertible operator on $X$. Hence, $\sigma(T) = \{0,\omega^{k+1}\}$. 
\vspace{1mm} 
Thus the operator $T$ has spectral decomposition (see \autoref{resultonnpotentdecomposition} and \autoref{thm:spectrumandeigenprojection}):
\[
T = 0\cdot P_{\{0\}}+ \omega^{k+1}P_{\{\omega^{k+1}\}}=\omega^{k+1}P_{\{\omega^{k+1}\}},
\]
with each projection $ P_{\{\lambda\}}$ given by:
\begin{align*}
& P_0 = I - T^{2k},\\
& P_{\{\omega^{k+1}\}} = \frac{1}{2k} \sum_{i=1}^{2k} \omega^{-i(k+1)} T^{i}. \tag{3}
\end{align*}
Since $X=\operatorname{Ker}(T)\bigoplus \operatorname{Ker}(T-\omega^{k+1}I)$, we note that the operator $T$ acts on the subspaces as follows: 
  \[
Tf =
\begin{cases}
0,        & f \in \operatorname{Ker}(T),\\
\omega^{k+1}f, & f \in \operatorname{Ker}(T-\omega^{k+1}I).
\end{cases}
\] 
A simple computation reveals that the Riesz projection $P_{\{\omega^{k+1}\}}$ acts on $X$ as follows:
  \[
P_{\{\omega^{k+1}\}}f =
\begin{cases}
0,        & f \in \operatorname{Ker}(T),\\
f, & f \in \operatorname{Ker}(T-\omega^{k+1}I),
\end{cases}
\] 
and hence $P_{\{\omega^{k+1}\}}=T^{2k}$. On the other hand,

  \[
(I-T^{2k})f =
\begin{cases}
f,        & f \in \operatorname{Ker}(T),\\
0, & f \in \operatorname{Ker}(T-\omega^{k+1}I),
\end{cases}
\] 
and therefore, $(I-T^{2k})$ is a projection onto the eigenspace associated with the eigenvalue zero and hence the eigenprojection.

\end{example}
\begin{example}
    Let $n \ge 2$ and $\omega=e^{\frac{2\pi i}{n-1}}$ be the $(n-1)$-th primitive root of unity. Define an operator $T \in \mathcal{B}(\mathbb{C}^{n-1})$ as:
\[
T=
\begin{pmatrix}
0      & 0      & 0      & 0      & \cdots & 0 \\
0      & 1      & 0      & 0      & \cdots & 0 \\
0      & 0      & \omega & 0      & \cdots & 0 \\
0      & 0      & 0      & \omega^{2} & \cdots & 0 \\
\vdots & \vdots & \vdots & \vdots & \ddots & \vdots \\
0      & 0      & 0      & 0      & \cdots & \omega^{\,n-2}
\end{pmatrix}.
\]
Then $T$ is an $n$-potent operator. Moreover, $\sigma(T)=\{0,1, \omega, \cdots, \omega^{n-2}\}$. Each eigenvalue is semisimple and the corresponding Riesz projections are the coordinate projections:
\[
P_{\{\lambda_j\}}=\operatorname{diag}(0,\dots,0,1,0,\dots,0),
\]
where $1$ is in the $j^{\text{th}}$ position.

\end{example}

\subsection{Projections in $\operatorname{comb}(T)$ for a $5$-potent operator $T$ and related discussions}\label{end-of-paper discussion}

We now provide a list of projections $P \in \operatorname{comb}(T)$ for a 5-potent operator $T$. Note that from \autoref{SpecOfnPotentOpProp}, $\sigma(T) \subseteq \{0,1,i,-1,-i\}$ with associated Riesz projections $P_{\{0\}}, P_{\{1\}},P_{\{i\}},P_{\{-1\}},P_{\{-i\}}$. By \autoref{thm:proj-conv-comb-result}, every projection $P \in \operatorname{comb}(T)$ is determined by a subset $S \subseteq \{1,i,-1,-i\}$, and is given by \[P=\sum_{\lambda \in S}P_{\{\lambda\}}.\] So, $\operatorname{comb}(T)$ contains exactly $2^4=16$ distinct projections. The following table lists several subsets $S$ and associated projections in $\operatorname{comb}(T)$. The other projections can be found similarly. 
\begin{table}[h]
\centering

\begin{tabular}{c|c}
$S\subseteq \{1,i,-1,-i\}$ & Projection in $\operatorname{comb}(T)$\\
\hline
$\emptyset$   & $0$  \\
$\{1\}$   & $P_{\{1\}}$  \\
$\{i\}$ & $P_{\{i\}}$ \\
$\{-1\}$ & $P_{\{-1\}}$ \\
$\{-i\}$ & $P_{\{-i\}}$ \\
$\{i,-1\}$ & $P_{\{i\}} + P_{\{-1\}}$ \\
$\{1,-i\}$ & $P_{\{1\}} + P_{\{-i\}}$ \\
$\{1,i,-1,-i\}$ & $I - P_{\{0\}}$\\
\end{tabular}
\caption{Projections in $\operatorname{comb}(T)$ for a $5$-potent operator}
\end{table}

We will compute the projection $P \in \operatorname{comb}(T)$ corresponding to the subset $\{i,-1\}$. We start by illustrating the computation of finding the Riesz projections associated with $P_{\{i\}}$ and $P_{\{-1\}}$ followed by finding the scalars $\beta_{\lambda}$. 

\medskip

Let $T \in \mathcal{B}(X)$ be a $5$-potent operator and hence $\sigma(T) \subseteq \{0,1,i,-1,-i\}$. We first compute the Riesz projection associated with the eigenvalues $\lambda=i,-1$. From equation~(\ref{eq:riesz-projection-eq-1}), we get 
\[P_{\{\lambda\}}=\frac{1}{2\pi i} \int_{\Gamma_{\lambda}} \left[ \frac{I}{z} + \bigg(\frac{1}{z(z^{4}-1)}\bigg)T^{4}+\sum_{k=1}^{3} \left( \frac{z^{3-k}}{z^{4} - 1} \right) T^k \right] \, dz,\]where $\Gamma_{\lambda}$ is the Cauchy contour enclosing only $\lambda$ in its interior. Evaluating the integral corresponding to $\lambda=i$ and $\lambda=-1$ we get 
\begin{align*}
    P_{\{i\}}&= -\frac{i}{4}T-\frac{1}{4}T^2+\frac{i}{4}T^3+\frac{1}{4}T^4,\\
    P_{\{-1\}} &= -\frac{1}{4}T+\frac{1}{4}T^2-\frac{1}{4}T^3+\frac{1}{4}T^4.
\end{align*}
Therefore, corresponding to the subset $S=\{i,-1\} \subseteq \{0,1,i,-1,-i\}$, the projection $P\in \operatorname{comb}(T)$ is given by:
\[P:= P_{\{i\}}+P_{\{-1\}}= \left(-\frac{1}{4}-\frac{i}{4}\right)T+\left(-\frac{1}{4}+\frac{i}{4}\right)T^3+\frac{1}{2}T^4.\]

\medskip
Conversely, suppose we are given the projection 
\begin{equation}\label{eq:proj-for-5-potent-in-comb(T)}
    P=\left(-\frac{1}{4}-\frac{i}{4}\right)T+\left(-\frac{1}{4}+\frac{i}{4}\right)T^3+\frac{1}{2}T^4.
\end{equation}
We compare the representation of $P$ with the following decomposition described in (\autoref{thm:proj-conv-comb-result}):
\[P= \sum_{\lambda \in \{1,i,-1,-i\}}\beta_{\lambda}P_{\{\lambda\}}, \quad \beta_{\lambda}=\sum_{k=1}^{4}a_k\lambda^k.\]

From equation~(\ref{eq:proj-for-5-potent-in-comb(T)}) we identify that $a_1=\left(-\frac{1}{4}-\frac{i}{4}\right), a_2=0, a_3=\left(-\frac{1}{4}+\frac{i}{4}\right), a_4=\frac{1}{2}$. We compute $\beta_{\lambda} = \sum_{j=1}^4 a_j \lambda^j$, corresponding to
		$\lambda\in\{0,1,i,-1,-i\}$:
        \begin{itemize}
            \item $\beta_{0}=0.$
            \item $\beta_{1}=\sum_{j=1}^{4}a_j=0.$
            \item $\beta_{i}=\sum_{j=1}^{4}a_ji^{j}=1.$
            \item $\beta_{-1}=\sum_{j=1}^{4}a_j(-1)^{j}=1.$
            \item $\beta_{-i}=\sum_{j=1}^{4}a_j(-i)^{j}=0.$
        \end{itemize}
        Therefore, by \autoref{thm:proj-conv-comb-result}, we have
        \[P=\left(-\frac{1}{4}-\frac{i}{4}\right)T+\left(-\frac{1}{4}+\frac{i}{4}\right)T^3+\frac{1}{2}T^4=P_{\{i\}}+P_{\{-1\}}.\]
This example illustrates how the projections $P \in \operatorname{comb}(T)$ arise associated to a subset $S \subseteq \sigma(T)$ and also conversely, given a projection $P \in \operatorname{comb}(T)$ how the coefficients $\beta_{\lambda}$ recover the corresponding subset. 

\section*{Acknowledgment}
The authors would like to thank Fernanda Botelho for valuable discussions and suggestions that improved this work.

\end{document}